\documentclass{amsart}

\usepackage{amsmath,amssymb}
\newtheorem{theorem}{Theorem}[section]
\newtheorem{lemma}[theorem]{Lemma}
\newtheorem{definition}[theorem]{Definition}
\newtheorem{proposition}[theorem]{Proposition}
\newtheorem{corollary}[theorem]{Corollary}
\newtheorem{example}[theorem]{Example}
\newtheorem{remark}[theorem]{Remark}

\newcommand\R{\mathbb R}
\newcommand\Z{\mathbb Z}
\newcommand{\Ha}{{\mathcal H}}
\newcommand{\bI}{{\bf I}}
\newcommand{\bN}{{\bf N}}
\newcommand{\bM}{{\bf M}}

\newcommand{\rd}{\R^d}

\newcommand{\intr}{\operatorname{Int}}

\newcommand{\spt}{\operatorname{spt}}

\newcommand{\lin}{\operatorname{Lin}}

\newcommand{\Tan}{\operatorname{Tan}}

\newcommand{\rank}{\operatorname{rank}}
\newcommand{\esssup}{\operatorname{ess\,sup}}

\newcommand{\sph}{S^{d-1}}

\newcommand{\bw}{\textstyle{\bigwedge\nolimits}}
\newcommand{\Haus}{\Ha^{d-1}}
\newcommand{\UPR}{{\mathcal U}_{\operatorname{PR}}}
\newcommand{\cT}{{\mathcal T}}

\newcommand{\ap}{\operatorname{ap}}

\newcommand{\llc}{\;\halfsq\;}
\newcommand{\lrc}{\;\ihalfsq\;}

\def\halfsq{\hbox{\kern1pt\vrule height 7pt\vrule width6pt height 0.4pt depth0pt\kern1pt}}
\def\ihalfsq{\hbox{\kern1pt \vrule width6pt height 0.4pt depth0pt
                   \vrule height 7pt \kern1pt}}
\begin{document}
\title{Legendrian cycles and curvatures}
\author{Jan Rataj}
\address{Charles University, Faculty of Mathematics and Physics, Sokolovsk\'a 83, 18675 Praha 8, Czech Republic}
\email{rataj@karlin.mff.cuni.cz}
\author{Martina Z\"ahle}
\address{Mathematical Institute, Friedrich-Schiller-University, D-07740 Jena, Germany}
\email{martina.zaehle@uni-jena.de}
\thanks{The second author was supported by grant DFG ZA 242/5-2}

\begin{abstract}
Properties of general Legendrian cycles $T$ acting in $\rd\times S^{d-1}$ are studied. In particular, we give short proofs for certain uniqueness theorems with respect to the projections on the first and second component of such currents: In general, $T$ is determined by its restriction to the Gauss curvature form - this result goes back to J. Fu -  and in the full-dimensional case also by the restriction to the surface area form. As a tool a version of the Constancy theorem for Lipschitz submanifolds is shown. 
\end{abstract}


\maketitle

\section{Introduction}
With a domain in $\rd$ with $C^2$-smooth boundary we can associate its normal cycle which is a closed $(d-1)$-current given by integration over the unit normal bundle of the domain boundary. The normal cycle applied to certain $(d-1)$-forms (Lipschitz-Killing forms) yields the total higher order mean curvatures (cf.\ \cite{SW}). 

Even with certain non-smooth (singular) domains we can associate normal cycles. We mention only a few cases here: (unions of) sets with positive reach \cite{Z86} (\cite{RZ01}), subanalytic sets \cite{Fu94,Ni11}, compact definable sets \cite{Be07}, Lipschitz domains with bounded curvatures \cite{RZ03}, d.c.\ domains \cite{PR12}.

The normal cycles are always closed currents (cycles) supported in $\rd\times\sph$ with the Legendrian property roughly saying that tangent vectors have both components orthogonal to the second component of the pair from the normal bundle. Such currents are called Legendrian (see Definition~\ref{LC}) and they have been considered by Fu \cite{Fu98,Fu04} in connection with geometric properties of their carrying sets.  Bernig \cite{Be06} studied some set operations such as Minkowski addition and projection of Legendrian currents.

The properties of a (general) Legendrian cycle do not necessarily reflect the geometric properties of the associated domain. Hence, an additional condition concerning the relation to the topological properties (local Gauss-Bonnet formula) is often imposed, see the second part of Section~\ref{S_LC}.

The aim of this note is to analyze properties of (general) Legendrian cycles, in particular, in connection with its projections to the first and second component. We present a short proof of the uniqueness theorem, due to Fu \cite{Fu89} and proved also in a different setting by Nicolaescu \cite{Ni11}, saying that any compactly supported Legendrian cycle is uniquely determined by its restriction to the Gauss curvature form. (Note that this uniqueness property is of particular importance since it makes it possible to introduce Legendrian cycles and curvature measures by approximation with smooth sets, cf.\ \cite{RZ05,PR12}.) We further consider the projection onto the first component. Using an integral representation of a Legendrian cycle (derived in \cite{RZ05}), we describe the first projection and, under a full-dimensionality assumption, we show an analogous result on uniqueness w.r.t. the first coordinate projection. We need here a version of the Constancy theorem for Lipschitz submanifolds which is a result of independent interest and is shown in Section~\ref{S-CT}.

\section{Legendrian cycles}  \label{S_LC}
Throughout the paper, we shall use the terminology and notation concerning the theory of currents as given in the Federer's book \cite{Fe69}. In particular, ${\bf F}_{m}^{\operatorname{loc}}(U)$, ${\bf I}_{m}^{\operatorname{loc}}(U)$ denote the space of locally flat, locally integer-multiplicity rectifiable $m$-currents in an open subset $U$ of some Euclidean space, respectively, see \cite[\S4.1.12,24]{Fe69}. The Euclidean scalar product of two vectors $u,v$ is denoted by $u\cdot v$ and $|u|=\sqrt{u\cdot u}$ is the norm of $u$.

\begin{definition} \label{LC}\rm
A {\it Legendrian cycle} in $\rd$ is an integer-multiplicity rectifiable current $T$ on $\R^{2d}$ with the following properties:
\begin{eqnarray}
&&\spt T\subset \rd\times\sph,\\
&&\partial T=0 \quad (T\text{ is a cycle}),\\
&&T\llc\alpha =0 \quad (T\text{ is Legendrian}),
\end{eqnarray}
where $\alpha$ is the contact form in $\R^{2d}$ acting as
$$\langle (u,v),\alpha(x,n)\rangle =u\cdot n,\quad u,v,x,n\in\rd.$$
\end{definition}

Any integer-multiplicity rectifiable current $T\in{\bf I}_{d-1}^{\operatorname{loc}}(\rd\times\rd)$ has an integral representation of the form
\begin{equation}  \label{repres}
T=(\Ha^{d-1}\llc W_T)i_T\wedge a_T,
\end{equation}
with a locally $(\Haus,d-1)$-rectifiable set $W_T\subset\rd\times\sph$, a unit simple measurable $(d-1)$-vector field $a_T$ on $W_T$ (prescribing an orientation of $W_T$) and an integer-valued integrable function $i_T$ on $W_T$ (index function), cf.\ \cite[\S4.1.28]{Fe69} or \cite[\S7.5]{KP08}.

Throughout the paper, $\pi_0,\pi_1$ will denote the two component projections in $\rd\times\rd$, i.e.,
$$\pi_0(x,n)=x,\quad \pi_1(x,n)=n.$$

\begin{theorem}[{\cite[Proposition 3, Lemma 2]{RZ05}}]  \label{T-tan}
Let $T$ be a Legendrian cycle in $\rd$. Then,
for $\Haus$-almost all $(x,n)\in W_T$, $\Tan^{d-1}(W_T,(x,n))$ is a $(d-1)$-dimensional subspace of $\rd\times\rd$ and there exists a positively oriented orthonormal basis $\{b_1(x,n),\ldots ,b_{d-1}(x,n),n\}$ of $\rd$ and numbers $\kappa_1(x,n),\ldots,\kappa_{d-1}(x,n)\in (-\infty,\infty]$ such that the vectors
$$a_i(x,n):=\left( \frac 1{\sqrt{1+\kappa_i^2(x,n)}}b_i(x,n), \frac {\kappa_i(x,n)}{\sqrt{ 1+\kappa_i^2(x,n)}}b_i(x,n)\right),\quad i=1,\ldots,d-1,$$
form an orthonormal basis of $\Tan^{d-1}(W_T,(x,n))$. (We set $\frac 1{\sqrt{1+\infty^2}}=0$ and $\frac {\infty}{\sqrt{1+\infty^2}}=1$.)
The numbers $\kappa_i(x,n)$ are uniquely determined, up to the order, and the subspace spanned by the vectors $b_j(x,n)$ belonging to a fixed value among the $\kappa_i(x,n)$ ($1\leq i\leq d-1$) is uniquely determined.
\end{theorem}

In analogy to sets with positive reach, we shall call the numbers $\kappa_i(x,n)$ {\it principal curvatures} and vectors $b_i(x,n)$ {\it principal directions} of $T$ at $(x,n)$.

If we orient the carrier $W_T$ by the unit $(d-1)$-vectorfield
$$a_T(x,n):=a_1(x,n)\wedge\cdots\wedge a_{d-1}(x,n),$$
there must exist an integrable integer-valued index function $i_T$ on $W_T$ such that \eqref{repres} holds. However, it turns out to be sometimes more convenient to work with another orientation. Let $\lambda(x,n)$ be the number of negative principal curvatures an $(x,n)$ (defined $\Ha^{d-1}$-almost everywhere on $W_T$), set $\tilde{a}_T(x,n):=(-1)^{\lambda(x,n)}a_T(x,n)$ and 
\begin{equation} \label{iota}
\iota_T(x,n):=(-1)^{\lambda(x,n)}i_T(x,n).
\end{equation} 
This particular form of the index function in \eqref{repres} will be used when working with the projection to the second component in Section~\ref{S-Gauss}.

\begin{remark} \rm
The ``carrying set'' $W_T$ of $T$ needs not be closed, whereas the support $\spt T$ is closed, but need not be locally $(\Haus,d-1)$-rectifiable for a general integer-multiplicity rectifiable current $T$. We do not know whether $\spt T$ must be locally $(\Haus,d-1)$-rectifiable for a Legendrian cycle $T$, or even whether it must have zero $d$-dimensional measure, cf.\ \cite{PR12} for the class of d.c.\ domains.
\end{remark}

\subsection*{Normal cycles}
A natural question arises, whether a Legendrian cycle is a normal cycle attached to some subset of $\rd$, as described in the Introduction for some classical set classes. To clarify this, we recall the definition of a normal cycle from \cite{PR12} (which is equivalent to the notion used in \cite{Fu94}).
A compact set $A\subset\rd$ {\it admits a normal cycle}, $N_A$, if there exists a compactly supported Legendrian cycle $N_A$ such that
\begin{enumerate}
\item[(i)] $\pi_0(\spt N_A)\subset\partial A$,
\item[(ii)] the image of $\spt N_A$ under $(x,n)\mapsto (n,x\cdot n)$ has $\Ha^d$-measure zero (in other words, the set of supporting hyperplanes of $N_A$ has $d$-dimensional measure zero),
\item[(iii)] $\langle N_A,\pi_1,-v\rangle (H_{v,t}\times\sph)=\chi(A\cap H_{v,t})$ for $\Ha^{d-1}$a.a. $(v,t)\in\sph\times\R$,
\end{enumerate}
where $H_{v,t}$ is the halfspace $\{ x\in\rd:\, x\cdot v\leq t\}$ and $\chi$ denotes the Euler-Poincar\'e characteristic (in the sense of singular homology). Note that the slice $\langle N_A,\pi_1,-v\rangle$ of the $(d-1)$-current $N_A$ under $\pi_1$ is a $0$-current, i.e., an integer-valued measure, and we can evaluate it at a Borel set. From (iii) one can derive special forms of the index function $\iota_T$ from \eqref{iota}, see \cite[Proposition~5]{RZ05} or \cite[Proposition~8.7]{PR12}.

The following simple example shows that not every compactly supported Legendrian cycle is a (multiple of a) normal cycle of some compact set.

\begin{example} \rm
Consider two line segments, $A=[(-1,0),(0,0)]$, $B=[(0,0),(1,0)]$ in the plane, and let $N_A$, $N_B$ be their normal cycles. Then, $N_A+N_B$ is a compactly supported Legendrian cycle which cannot be a normal cycle of a compact set. Indeed, if there would be a compact set $C$ with $N_C=N_A+N_B$ then condition (iii) yields
$$\chi(C\cap H_{v,t})=\chi(A\cap H_{v,t})+\chi(B\cap H_{v,t})$$
for almost all $(v,t)\in\sph\times\R$. One derives that $C$ must be included in $A\cup B$ and that $C$ must contain both segments, $A$ and $B$, thus, $C=A\cup B$, but then, of course, the global Euler characteristic disagrees, so $N_C\neq N_A+N_B$.
\end{example}

\section{Constancy theorem for Lipschitz manifolds} \label{S-CT}
The Constancy theorem for currents says in principle that a $k$-dimensional current without boundary in a $k$-dimensional linear space equals a multiple of the canonical current given by integration w.r.t.\ $k$-dimensional Lebesgue measure (see \cite[p.~357]{Fe69}). Its generalized version concerns a flat $k$-current without boundary living in a $k$-dimensional oriented connected $C^1$-manifold, see \cite[\S4.1.31]{Fe69}. We shall need in the sequel a version of Constancy theorem for oriented Lipschitz manifolds. Both the statement and proof are similar to those for $C^1$-manifolds. Nevertheless, we present them, since we did not find them in the literature.

Recall that $M\subset\rd$ is an {\it orientable $m$-dimensional Lipschitz submanifold} if
there exists an atlas of Lipschitz parametrizations, i.e. $\{f_i:U_i\rightarrow \rd\}_{i=1}^\infty$ for  open $U_i\subset\mathbb{R}^m$ and bi-Lipschitz injections $f_i$ with $M=\bigcup_{i=1}^\infty f_i(U_i)$, such that $\det\big(D(f_j^{-1}|_{f_j(U_j)}\circ f_i)\big)>0$ a.e. on $U_i\cap U_j$. The $m$-vector field $\xi$ on $M$ determined for a.e. $x=f_i(u)\in f_i(U_i)$ by
$$\xi(x):=\frac{\bigwedge_m Df_i(u)( e_1\wedge\ldots\wedge e_m)}{\big|\bigwedge_m Df_i(u)( e_1\wedge\ldots\wedge e_m)\big|}$$
orients $M$, where $e_1,\ldots,e_m$ is a fixed orthonormal basis in $\mathbb{R}^m$. (Note that the multivector field $\xi$ defined a.e.\ on $M$ clearly needs not be continuously extendable to $M$.)

\begin{theorem}[Constancy theorem for Lipschitz submanifolds]\label{constancytheoremLipschitz}
Suppose $M$ is an $m$-dimensional connected Lipschitz submanifold of an open set $U\subset\rd$ with orienting $m$-vector field $\xi$. Then for any current $T\in {\bf F}_m^{\operatorname{loc}}(U)$ with
$$\spt T\setminus M \text{ closed relative to }U,\quad \spt\partial T\subset U\setminus M,$$
there exists a real number $r$ such that
$$\spt(T-r(\Ha^m\llc M)\wedge \xi)\subset U\setminus M.$$
If $T$ is locally rectifiable then $r$ is an integer.
\end{theorem}

\begin{corollary}
Let $T\in {\bf F}_m^{\operatorname{loc}}(U)$ and let $M$ be a connected $m$-dimensional Lipschitz submanifold of $U$ with orienting $m$-vector field $\xi$ such that $\spt T\subset\overline{M}$. If 
$$\spt\partial T\subset \overline M\setminus M$$
(in particular, if $T$ is a cycle), then
$$T=r (\Ha^m\llc M)\wedge\xi$$
for some $r\in\R$.
\end{corollary}

\begin{proof}[Proof of Theorem~\ref{constancytheoremLipschitz}]
We present a proof which uses the idea of \cite[\S4.1.31]{Fe69}) for $C^1$ submanifolds.

Let $\{f_i:U_i\rightarrow \rd\}_{i=1}^\infty$ be an atlas of bi-Lipschitz parametrizations of $M$. Choose $C^\infty$-functions $\gamma_i:\R^d\to[0,1]$ with compact supports $\spt\gamma_i\subset f_i(U_i)$ and such that the open sets $W_i:=\intr\{x:\, \gamma_i(x)=1\}$ cover $M$ (it is not difficult to see that such functions always exist). 

Then, $T\llc\gamma_i\in{\bf F}_m(\rd)$ and 
$$\spt\partial (T\llc\gamma_i)\subset M\setminus W_i.$$
Since $f_i$ is bi-Lipschitz, $f_i^{-1}$ is proper and we can define the push-forward (cf.\ \cite[Lemma~7.4.3]{KP08}
$$(f_i^{-1})_{\#}(T\llc\gamma_i)\in {\bf F}_m(U_i)$$
with $\spt \partial (f_i^{-1})_{\#}(T\llc\gamma_i)\subset U_i\setminus f_i^{-1}(W_i)$.
Applying the Constancy theorem \cite[\S4.1.7]{Fe69}, we get that for some $r_i\in\R$,
$$\spt((f_i^{-1})_{\#}(T\llc\gamma_i)-r_i({\mathbb E}_m\llc U_i))\subset U_i\setminus f_i^{-1}(W_i)$$
(here ${\mathbb E}_m$ denotes the $m$-current ${\mathcal L}^m\wedge(e_1\wedge\cdots\wedge e_m)$ in $\R^m$ with canonical basis $\{ e_1,\ldots,e_m\}$). Thus,
$$\spt((T\llc\gamma_i)-r_i(f_i)_{\#}({\mathbb E}_m\llc U_i))= \spt((T\llc\gamma_i)-r_i({\mathcal H}^m\llc M)\wedge\xi)\subset M\setminus W_i$$
(we have used that $(f_i)_{\#}({\mathbb E}_m\llc U_i)=({\mathcal H}^m\llc M)\wedge\xi$ by the area formula for currents, \cite[\S4.1.25]{Fe69}, and the choice of the orienting $m$-vector field $\xi$).
If $\Ha^m(W_i\cap W_j)>0$ for two indices $i,j$ then it follows that $r_i=r_j$. Thus, since $M$ is connected and covered by the $W_i$'s, we infer that all the $r_i$'s equal a constant $r\in \R$ and $\spt(T-r(\Ha^m\llc M)\wedge\xi)\subset U\setminus M$.

The last statement of the theorem follows from \cite[\S4.1.28]{Fe69}.
\end{proof}

\section{Gauss curvature form}  \label{S-Gauss}
The Gauss curvature form $\varphi_0$ (called also Lipschitz-Killing $(d-1)$-form of order $0$) is a differential $(d-1)$-form on $\rd\times\rd$ which can be given by
$$\varphi_0(x,n)=(d\omega_d)^{-1}(\pi_1)^{\#}(n\lrc\Omega_d),\quad (x,n)\in\rd\times\rd$$
(with $\omega_d=\pi^{d/2}/\Gamma(\frac d2+1)$), where we use the formalism of interior multiplication of multivectors and multicovectors from \cite{Fe69} and $\Omega_d$ stands for the standard volume form in $\rd$. More explicitely, we can write
$$\langle a_1\wedge\dots\wedge a_{d-1},\varphi_0(x,n)\rangle=(d\omega_d)^{-1}\langle\pi_1(a_1)\wedge\dots\wedge\pi_1(a_{d-1})\wedge n,\Omega_d\rangle,$$
$a_1,\ldots,a_{d-1}\in\rd\times\rd$ (cf.\ \cite{Z86}). 

Given a unit vector $n$, let $n^*$ denote the unique unit simple $(d-1)$-vector associated with the $(d-1)$-subspace $n^\perp$ orthogonal to $n$ such that 
\begin{equation} \label{n*}
\langle n^*\wedge n,\Omega_d\rangle =1.
\end{equation} 

The push-forward $(\pi_1)_{\#}T$ is a $(d-1)$-current without boundary on the unit sphere, hence, the Constancy theorem can be applied and we get the following result. 

\begin{theorem}[{\cite[Theorem 4]{RZ05}}]  \label{T-sph}
For any compactly supported Legendrian cycle $T$ we have
\begin{enumerate}
\item[{\rm (i)}] $(\pi_1)_{\#}T=T(\varphi_0)(\Haus\llc\sph)\wedge n^*$,
\item[{\rm (ii)}] $T(\varphi_0)=\sum_{x\in\rd} \iota_T(x,n)$ for almost all $n\in\sph$.
\end{enumerate}
\end{theorem}

Note that the number $T(\varphi_0)$ is called Euler characteristic of $T$ in \cite[Proposition~1.4]{Be06}.

To any given $\Ha^{d-1}$-summable function $i:\rd\times\sph\to\Z$, there exists, of course, at most one compactly supported Legendrian cycle $T$ with index function $\iota_T=i$.
The following uniqueness theorem due to J. Fu is a stronger version of this statement above; it says that a compactly supported Legendrian cycle is uniquely determined by its restriction to the Gauss curvature form.

\begin{theorem}[{\cite[Theorems 1.1, 4.1]{Fu89}}]  \label{T-uniq}
Let $T$ be a compactly supported Legendrian cycle such that $T\llc\varphi_0=0$. Then $T=0$.
\end{theorem}

An important step in the proof will be based on a ``separation theorem'' due to Hardt \cite{Ha77} which reads as follows. A current $T\in\bI_k(\rd)$ is called {\it indecomposable} (see \cite[\S4.2.25]{Fe69}) if there exists no $S\in\bI_k(\rd)$ with $S\neq 0\neq T-S$ and $\bN(T)=\bN(S)+\bN(T-S)$. (Here $\bN(T)=\bM(T)+\bM(\partial T)$ and $\bM(T)$ is the mass norm of $T$.)

\begin{theorem}[Hardt {\cite[Theorem~1]{Ha77}}]  \label{T_Hardt}
Let $T=(\Ha^k\llc W_T)\wedge\xi\in\bI_k(\rd\times\rd)$ be indecomposable and assume that $\xi\in\rd\times\{0\}$ $\Ha^k$-almost everywhere. Then $\spt T\subset\rd\times\{a\}$ for some $a\in\rd$.
\end{theorem}

\begin{proof}[Proof of Theorem~\ref{T-uniq}]
Using \cite[Theorem 5]{RZ05}, we have for any smooth function $g$ on $\rd\times\sph$
$$(T\llc\varphi_0)(g)=T(g\varphi_0)=(d\omega_d)^{-1}\int_{W_T}\iota_T\:g\:K_0\, d\Ha^{d-1},$$
where
$$K_0(x,n)=\frac{\kappa_i(x,n)\cdot\ldots\cdot\kappa_{d-1}(x,n)}
{
\sqrt{\vphantom{\kappa_{d-1}^2(x,n)} 1+\kappa_{1}^2(x,n)}
\cdot\ldots\cdot
\sqrt{1+\kappa_{d-1}^2(x,n)}
}$$
is the generalized Gauss curvature defined $\Ha^{d-1}$-almost everywhere on $W_T$ and $\kappa_i(x,n)$ are the principal curvatures at $(x,n)$ from Theorem~\ref{T-tan}.
The assumption $T\llc\varphi_0=0$ thus implies $K_0=0$ almost everywhere on $W_T$. In particular, denoting $f=\pi_1|W_T$, we have $\rank Df <d-1$ almost everywhere. 

Assume, for the contrary, that $T\neq 0$, and let
$$m:=\esssup \rank Df <d-1.$$
We can then assume, without loss of generality (removing a subset of $(d-1)$-dimensional measure zero), that at all $(x,n)\in W_T$, 
\begin{enumerate}
\item[(i)] $\Tan^{d-1}(W_T,(x,n))$ has the basis given in Theorem~\ref{T-tan} with principal curvatures $\kappa_i(x,n)=0$ whenever $i>m$,
\item[(ii)] $\Ha^{d-1-m}(f^{-1}\{n\})>0$ (cf.\ \cite[\S3.2.32]{Fe69}),
\item[(iii)] either $\ap J_mf(x,n)=0$ or $\Tan^m(f(W_T),n)=\lin\{ b_1(x,n),\ldots,b_m(x,n)\}$ (cf.\ \cite[\S3.2.22]{Fe69}).
\end{enumerate}
(In the case $d=3$ and $m=1$ it might happen that, in order to preserve the prescribed orientation, we have to choose $\kappa_1=0$ instead of $\kappa_2=0$, and the role of the indices must be exchanged in the whole proof.)

Assume that $(x,n)\in W_m:=\{(x,n)\in W_T:\, \rank Df(x,n)=m\}$. Then, the subspaces
$U(n):=\lin \{ b_i(x,n):\, 1\leq i\leq m\}$ and $V(n):=\lin \{ b_i(x,n):\, m+1\leq i\leq d-1\}$ do not depend on $x$ for which $(x,n)\in W_m$, due to (iii). 

Let $p$ be the orthogonal projection from $\rd$ onto an $m$-dimensional subspace $L_m$ of $\rd$, and let $\Omega_m$ denote the volume form in $L_m$. Due to the choice of $m$, we have 
$$T\llc (p\circ\pi_1)^{\#}\Omega_m\neq 0$$
for some fixed subspace $L_m$. Slicing $T$ by $p\circ\pi_1$, we get for the masses $\bM$ of the currents (cf.\ \cite[Theorem~4.3.2]{Fe69})
$$0\neq{\bM}(T\llc (p\circ\pi_1)^{\#}\Omega_m)=\int_{L_m}{\bM}\langle T,p\circ\pi_1,y\rangle\, \Ha^m(dy).$$
Hence, there exists a measurable subset $Y$ of $L_m$ of positive measure such that for $y\in Y$:
\begin{enumerate}
\item[(iv)] $\langle T,p\circ\pi_1,y\rangle\neq 0$ is a cycle (cf.\ \cite[p.~437, line -6]{Fe69});
\item[(v)] (see \cite[\S4.3.8]{Fe69})
$$\langle T,p\circ\pi_1,y\rangle =(\Ha^{d-1-m}\llc(f^{-1}(p^{-1}\{y\})))\wedge\zeta(x,n),$$
where for $(x,n)\in W_m\cap f^{-1}(p^{-1}\{y\})$,
\begin{eqnarray*}
\zeta(x,n)&=&\iota_T(x,n)a_T(x,n)\llc(p\circ\pi_1)^{\#}\Omega_m /\ap J_m(p\circ f)(x,m)\\
&=&\iota_T(x,n)(b_{m+1}(x,n),0)\wedge\dots\wedge (b_{d-1}(x,n),0).
\end{eqnarray*}
\end{enumerate}
Take a point $y\in Y$, let $R$ be an indecomposable component of $\langle T,p\circ\pi_1,y\rangle$, and take a point $(x,n)\in W_m\cap\spt R$. Using Theorem~\ref{T_Hardt}, we get
$$\spt R\subset f^{-1}\{n\}.$$
Denoting by $q$ the orthogonal projection of $\rd\times\rd$ to $(V(n)\times\{0\})^\perp$, we have from the description above
$$\zeta\llc q=0\quad \Ha^{d-1-m}-a.e.\text{ on }f^{-1}\{n\},$$ 
which implies that $q|\spt R$ is constant (we use again Theorem~\ref{T_Hardt}, cf.\ also an equivalent formulation in \cite[p.~1]{Ha77}).
Hence, $\spt R$ is contained in the affine $(d-1-m)$-subspace $(x,n)+V(n)\times\{0\}$. The Constancy theorem yields that $R$ must be a multiple of the Lebesgue measure. Since $T$ (and, hence, also $R$) is compactly supported, we get $R=0$, a contradiction.
\end{proof}

\begin{remark}  \rm
Let $\mu$ be the measure on $\rd\times\sph$ defined by
$$\int g(x,n)\, d\mu(x,n)=\int_{\sph}\int_{\pi_1^{-1}\{n\}}g(x,n)\, \Ha^0(d(x,n))\, \Ha^{d-1}(dn).$$
The uniqueness theorem says that if two compactly supported Legendrian cycles $T=(\Ha^{d-1}|W_T)\wedge \iota_Ta_T$ and $S=(\Ha^{d-1}|W_S)\wedge \iota_Sa_S$ fulfill that $\iota_T=\iota_S$ $\mu$-almost everywhere, then $S=T$.
\end{remark}

\section{Projection onto the first component and surface area}  \label{S_first}
If $X\subset\rd$ has positive reach, the pushforward $(\pi_0)_{\#}N_X$ of its normal cycle $N_X$ is the current given by integration w.r.t.\ $\Haus$ over the oriented topological boundary $\partial X$. Our next aim is to study under which conditions a Legendrian cycle 
pertains this property.

Let $T$ be a Legendrian cycle with integral representation \eqref{repres}.
We partition $\pi_0(W_T)=W_1\cup W_2\cup W_3$ into three sets so that:
\begin{eqnarray*}
W_1&=&\{ x\in\rd:\, \exists n\in\sph,\,\pi_0^{-1}\{ x\}\cap W_T=\{ (x,n)\}\},\\
W_2&=&\{ x\in\rd:\, \exists n\in\sph,\,\pi_0^{-1}\{ x\}\cap W_T=\{(x,-n),(x,n)\}\},\\
W_3&=&\pi_0(W_T)\setminus (W_1\cup W_2).
\end{eqnarray*}

\begin{lemma}
Let $W_T,W'_T$ be two carrying sets of a compactly supported Legendrian cycle $T$, and assume that the index function $i_T$ is nonzero $\Haus$-almost everywhere on both $W_T$ and $W'_T$. Let $W_i,W_i'$ be the corresponding components of $\pi_0(W_T)$, $\pi_0(W'_T)$ defined above, $i=1,2,3$. Then
$$\Haus(W_T\triangle W'_T)=0,\quad \Haus(W_i\triangle W'_i)=0,\quad i=1,2,3.$$
(The symbol $\triangle$ denotes the symmetric difference.)
\end{lemma}

\begin{proof}
We have clearly $\| T\|(W_T\triangle W'_T)=0$ for the measure $\| T\|$ associated with the current $T$. Since $\Haus$ is absolutely continuous w.r.t.\ $\|T\|$ by our assumption, the first assertion follows. For the second assertion, observe that
$W_i\triangle W_i'\subset\pi_0(W_T\triangle W'_T)$, $i=1,2$.
\end{proof}

We define the mapping $\nu:W_1\to\sph$ assigning to each $x\in W_1$ the unique unit vector $n$ such that $(x,n)\in W_T$. We further extend $\nu$ to a measurable mapping on $W_1\cup W_2$ such that $(x,\nu(x))\in W_T$.
Recall the definition of the mapping $n\mapsto n^*$ from \eqref{n*}.

\begin{proposition}  \label{P-pi0}
If $T$ is a Legendrian cycle then
\begin{eqnarray*}
(\pi_0)_{\#}T&=&(\Haus\llc W_1)\wedge i_T(x,\nu(x))\nu(x)^*\\ 
&&+(\Haus\llc W_2)\wedge i_T(x,\nu(x))\nu(x)^*\\ 
&&+(\Haus\llc W_2)\wedge i_T(x,-\nu(x))(-\nu(x))^*.
\end{eqnarray*}
\end{proposition}

\begin{proof}
Denote by $f=\pi_0|W_T$ the restriction of $\pi_0$ to $W_T$.
Using the Area theorem for currents (see \cite[\S4.1.25]{Fe69}), we get
$$(\pi_0)_{\#}T=(\Haus\llc\pi_0W_T)\wedge\xi$$
with
$$\xi(x)=\sum_{n:\, (x,n)\in W_T}\frac{(\bw_{d-1}\pi_0)i_T(x,n)a_T(x,n)}  {J_{d-1}f(x,n)}$$
if $J_{d-1}f(x,n)> 0$ and $\xi(x)=0$ otherwise.
We shall show that $J_{d-1}f$ vanishes $\Haus$-almost everywhere on $\pi_0^{-1}W_3$. Let $(x,n)\in \pi_0^{-1}W_3$ be a regular point in the sense of Theorem~\ref{T-tan} and assume, for the contrary, that $J_{d-1}f(x,n)>0$. We have
$$J_{d-1}f(x,n)=|(\bw_{d-1}\pi_0)a_T(x,n)|$$
which is positive if and only if all the principal curvatures $\kappa_i(x,n)$ are finite by Theorem~\ref{T-tan}, and all the principal directions $b_i(x,n)$ belong to $\Tan^{d-1}(\pi_0W_T,x)$, $i=1,\ldots,d-1$. As $x\in W_3$ there must be another unit vector $n'$ linearly independent of $n$ and such that $(x,n')$ lies also in $W_T$. We may assume that $(x,n')$ is also regular in the sense of Theorem~\ref{T-tan} and that $J_{d-1}f(x,n')>0$ (since other points project to an $\Haus$-zero set under $\pi_0$). But then, by the same reason, also the the principal directions $b_i(x,n')\in\Tan^{d-1}(\pi_0W_T,x)$, $i=1,\ldots,d-1$. Since $n$ and $n'$ are linearly independent, $\Tan^{d-1}(\pi_0W_T,x)$ contains two different hyperplanes, which can happen only at points of $\Haus$-measure zero in an $\Haus$-rectifiable set (see \cite[\S3.2.19]{Fe69}). Thus, $J_{d-1}f$ vanishes $\Haus$-almost everywhere on $\pi_0^{-1}W_3$ and $(\pi_0)_{\#}T$ is carried by $W_1\cup W_2$ only.

It remains to show that $(\bw_{d-1}\pi_0)a_T(x,n)=n^*$ for $\Haus$-almost all $(x,n)\in W_1\cup W_2$. But this follows directly from Theorem~\ref{T-tan} for regular points $(x,n)$, and the proof is complete.
\end{proof}

If $X\subset\rd$ belongs to the class $\UPR$ of unions of sets with positive reach (see \cite{RZ01}) then the index function of the normal cycle $N_X$ equals $1$ except for a set which projects under $\pi_0$ to an $\Haus$-zero set. This, of course, does not hold for a general Legendrian cycle (take an $\UPR$ set and multiply each connected component of its normal cycle by an arbitrary integer). Nevertheless, we can get at least some results about constancy of the index function on components, under additional assumptions.

\begin{theorem}   \label{T-LM}
Let $M\subset\rd$ be an orientable Lipschitz submanifold of dimension $d-1$ and let $T$ be a Legendrian cycle with with $\Haus(W_2)=0$ and $W_1\subset M$. Then, for any connected component $C$ of $M$, there exists an integer $i_C$ such that
$$(\pi_0)_{\#}T\llc C=(\Haus\llc W_1)\wedge i_C n(x)^*.$$
Consequently, $i_T$ is constant almost everywhere on each connected component of $\pi_0^{-1}W_1$.
\end{theorem}

\begin{proof}
Follows from Proposition~\ref{P-pi0} and Theorem~\ref{constancytheoremLipschitz}.
\end{proof}

\begin{remark} \rm
The assumption that $W_1$ is contained in a Lipschitz submanifold cannot be relaxed. Consider two smooth simple closed curves $\gamma_1,\gamma_2:[0,2]\to\R^2$ such that $\gamma_1=\gamma_2$ on $[0,1]$ and $\gamma_1\neq\gamma_2$ on $(1,2)$, and let $T_1,T_2$ be the normal cycles associated with $\gamma_1,\gamma_2$, respectively, with normals pointing outward from the circumscribed region. Then $T=T_1+T_2$ is a Legendrian cycle with $W_1=\gamma_1[0,2]\cup\gamma_2[0,2]$ connected, $i_T(x,n(x))=2$ if $x\in\gamma_1[0,1]$ and $i_T(x,n(x))=1$ if $x\in\gamma_1(1,2)\cup\gamma_2(1,2)$ ($n(x)$ is the unique outer normal vector at $x$).
\end{remark}

In Section~\ref{S-Gauss}, we introduced the Gauss curvature form connected with the projection onto the second component. Now we shall consider the surface area form $\varphi_{d-1}$ instead, which corresponds to the projection onto the first component:
$$\varphi_{d-1}(x,n)=(\pi_0)^{\#}(n\lrc\Omega_d),\quad (x,n)\in\rd\times\rd.$$
Unlike Theorem~\ref{T-uniq}, we cannot expect that $T\llc\varphi_{d-1}$ determines a Legendrian cycle $T$, since $T$ might be ``lower-dimensional'' (it may happen that $\Ha^{d-1}(\pi_0(\spt T))=0$, consider e.g.\ the normal cycle of a smooth curve in $\R^3$). We shall introduce, therefore, a ``full-dimensionality'' assumption under which we get uniqueness.

We say that two currents $S,T$ on an open set $U\subset\rd$ are {\it orthogonal} if the associated measures $\|S\|$ and $\|T\|$ are orthogonal, i.e., if there exists a Borel set $B\subset U$ such that 
$$S\llc{\bf 1}_B=0=T\llc{\bf 1}_{U\setminus B}.$$ 
Let further $\rho:\, \rd\times\sph\to\rd\times\sph$ be the reflection
$$\rho:\, (x,n)\mapsto(x,-n).$$

\begin{definition} \rm
We shall call a Legendrian cycle $T$ {\it full-dimensional} if $T$ and $\rho_{\#}T$ are orthogonal.
\end{definition}

It is not difficult to show that $T$ is full-dimensional iff it has a career $W_T$ from \eqref{repres} such that
\begin{equation}  \label{fd}
\Haus (W_T\cap\rho(W_T))=0.
\end{equation}

\begin{theorem}  \label{T-uniq_2}
Let $T$ be a full-dimensional Legendrian cycle such that $T\llc\varphi_{d-1}=0$. Then $T=0$.
\end{theorem}

\begin{proof}
The proof is similar to that of Theorem~\ref{T-uniq}. Let $g=\pi_0|W_T$. The assumption $T\llc\varphi_{d-1}=0$ implies that $\rank Dg<d-1$ almost everywhere. We assume $T\neq 0$ for the contrary, and denote
$$m:=\esssup\rank Dg \quad(<d-1).$$
Removing a set of points of $\Ha^{d-1}$-measure zero, we may assume that for all $(x,n)\in W_T$, the tangent space $\Tan^{d-1}(W_T,(x,n))$ has the form given in Theorem~\ref{T-tan} with $\kappa_i(x,n)=\infty$ whenever $i>m$, that $\Ha^{d-1-m}(g^{-1}\{x\})>0$ and, denoting $W_m:=\{(x,n)\in W_T:\, \rank Dg(x,m)=m\}$ (which is a set of positive $\Ha^{d-1}$-measure), that 
$$\Tan^m(g(W_T),x)=U(x):=\lin\{b_1(x,m),\ldots,b_m(x,m)\},\quad (x,n)\in W_m.$$
Again, there exists a projection $p$ of $\rd$ onto an $m$-dimensional subspace $L_m$ (with volume form $\Omega_m$) such that
$$T\llc(p\circ\pi_0)^{\#}\Omega_m\neq 0,$$
and a subset $Y\subset L_m$ of positive $\Ha^m$ measure such that
$$0\neq\langle T,p\circ g,y\rangle =(\Ha^{d-1-m}\llc(g^{-1}(p^{-1}(y)))\wedge\zeta,\quad y\in Y,$$
where
$$\zeta(x,n)=i_T(x,n)(0,b_{m+1}(x,n))\wedge\dots\wedge(0,b_{d-1}(x,m)).$$
Taking $y\in Y$, an indecomposable component $R\neq 0$ of $\langle T,p\circ g,y\rangle$ and $(x,n)\in W_m\cap\spt R$, the result of Hardt \cite[Theorem~1]{Ha77} implies that $\spt R\subset g^{-1}\{x\}$. 
Let $V(x)$ be the $(d-m)$-subspace of $\rd$ orthogonal to $U(x)$ and let $q$ be the orthogonal projection from $\rd\times\rd$ onto the orthogonal complemet to $\{0\}\times V(x)$. Since $R\llc q=0$, $q|\spt R$ is constant by the result of Hardt \cite{Ha77} and we get that $\spt R$ is contained in the $(d-1-m)$-dimensional submanifold $M:=(x,n)+(\{x\}\times \sph\cap V(x))$ Applying now the Constancy theorem (Theorem~\ref{constancytheoremLipschitz}), we obtain that $R$ is a multiple of $(\Ha^{d-1-m}\llc M)\wedge\xi$ for some orienting vertorfield $\xi$ of $M$. In particular, we get
$\rho_{\#}R=\pm R$ for the reflection $\rho$ and, consequently,
$$\int_Y\int_{(p\circ g)^{-1}\{y\}}{\bf 1}_{\rho(W_T)}(x,n)\, \Ha^{d-1-m}(d(x,n))\, \Ha^m(dy)>0.$$
The Coarea formula implies that $\int_{W_T} {\bf 1}_{\rho(W_T)}(x,n)\, \Ha^{d-1}(d(x,n))>0$, which contradicts the full-dimensionality assumption.
\end{proof}

\begin{remark}\rm
Theorem~\ref{T-uniq_2} and Proposition~\ref{P-pi0} imply that a full-dimensional Legendrian cycle is determined by its restriction to the set $W_1$ of pairs of points with a unique unit normal.
\end{remark}

\section{Some further properties of Legendrian cycles}

Theorem~\ref{T-LM} says that the index function $i_T$ of a Legendrian cycle $T$ is constant almost everywhere on each connected component of $\pi_0^{-1}W_1$, provided that $W_1$ is contained in an orientable Lipschitz submanifold of dimension $d-1$. Without the last assumption, we can still obtain constancy of the index function on certain parts of $W_1$.

Let $T$ be a Legendrian cycle given in the form \eqref{repres}. Let $R_T$ be the set of all points of $\rd$ such that 
$$\cT(x):=\Tan^{d-1}(\pi_0W_T,x)$$ 
is a $(d-1)$-dimensional subspace. Of course, $R_T$ is $(\Ha^{d-1},d-1)$-rectifiable and none of $R_T$ and the function $x\mapsto\cT(x)$ depend on the chosen representant $W_T$ carrying the cycle $T$. From the Legendrian property, in particular, from Theorem~\ref{T-tan}, one easily derives that for almost all $x\in R_T$, $(x,n)\in W_T$ can happen only if $n\perp\cT(x)$. 

As a corollary of Proposition~\ref{P-pi0}, we obtain that
\begin{equation}  \label{supp_proj}
\spt (\pi_0)_{\#}T\subset \overline{R_T}.
\end{equation}

Recall the notation $W_i$, $i=1,2,3$, from the beginning of Section~\ref{S_first}.

\begin{proposition}
Let $T$ be a Legendrian cycle with $\Ha^{d-1}(W_2)=0$ and $M$ a connected orientable $(d-1)$-dimensional Lipschitz submanifold of $\rd$ such that
\begin{enumerate}
\item[{\rm (i)}] $M\subset\overline{R_T}$,
\item[{\rm (ii)}] $\overline{R_T}\setminus M$ is closed.
\end{enumerate}
Then the index function $i_T$ is constant almost everywhere on $\pi_0^{-1}M$.
\end{proposition}

\begin{proof}
Assumptions (i) and (ii) together with \eqref{supp_proj} imply that $\spt (\pi_0)_{\#}T\setminus M$ is closed, hence, we may apply Theorem~\ref{constancytheoremLipschitz}.
\end{proof}

\subsection*{Support of a Legendrian cycle}
If $T$ is the normal cycle of the graph of a sufficiently regular locally Lipschitz function (convex, semiconvex, or even delta-convex) then $\spt T$ is contained in the graph of the Clarke normal bundle of the function. This follows from a result on Monge-Amp\`ere functions \cite[Theorem~2.2]{Fu89}. In the sequel, we show an analogous result for a general Legendrian cycle which is full-dimensional in a stronger sense.

We call a Legendrian cycle $T$ {\it strongly full-dimensional} if $\spt T\cap\rho(\spt T)=\emptyset$. Note that, in particular, the normal cycle of a Lipschitz domain must be strongly full-dimensional whenever it exists (cf.\ \cite{RZ03}).

Let $T$ be a strongly full-dimensional compactly supported Legendrian cycle given in the form \eqref{repres}, let
$W_R$ be the closure of the set of all $(x,n)\in W_T$ such that $x\in R_T$ and $n\perp\cT(x)$, and let $C_T$ denote the set of all $(x,n)\in\rd\times\sph$ such that $n$ lies in the spherical convex hull of all points $u\in\sph$ such that $(x,u)\in W_R$.

\begin{proposition}
For any strongly full-dimensional compactly supported Legendrian cycle $T$, 
$$\spt T\subset C_T.$$
\end{proposition}

\begin{proof}
Assume again that for all $(x,n)\in W_T$, the approximate tangent space $\Tan^{d-1}(W_T,(x,n))$ has the structure given in Theorem~\ref{T-tan} and let $g=\pi_0|W_T$. Let $W_k$ denote the set of all $(x,n)\in W_T$ where $\rank Dg(x,n)=k$ (equivalently, $k$ principal curvatures are finite and the remaining $d-1-k$ infinite). Denote
$$T_k:=T\llc {\bf 1}_{W_k},\quad k=0,\ldots, d-1.$$
We shall show by induction that $\spt T_k\subset C_T$ for all $k$. For $k=d-1$, this follows immediately from the definition of $C_T$ since $T_{d-1}$ is supported by $W_R$. Let further $m<d-1$ be given and assume that $\spt T_k\subset C_T$, $k>m$. We shall show that $\spt T_m\subset C_T$ as well.

Consider the Legendrian current $S:=T\llc{\bf 1}_{(C_T)^c}$ (which, of course, need not be a cycle). We shall show that $S=0$. Assume, for the contrary, that $S\neq 0$. We proceed exactly as in the proof of Theorem~\ref{T-uniq_2} and we find a non-zero indecomposable integer-rectifiability $(d-1-m)$-current $R$ supported in $M:=\{x\}\times(\sph\cap V(x))$, where $V(x)$ is a $(d-m)$-dimensional subspace of $\rd$. Of course, $\partial S$ is supported in $C_T$. Choose some orienting $(d-1-m)$-vectorfield $\eta$ of $M$. 
Let $M'$ be the complement (in $M$) of the spherical convex hull of $C_T\cap M$. Applying Theorem~\ref{constancytheoremLipschitz} to $R$, we find that
$$\spt (R-r(\Ha^{d-1-m}\llc M)\wedge\eta)\subset M\setminus M'$$
for some $r\in\R$. If $r\neq 0$ then $M'\subset\spt R\subset\spt T$, but this would contradict the strong full-dimensionality assumption. Thus $r=0$, $\spt R\subset M\setminus M'$, which is contained in the spherical convex hull of $C_T(x)$. By the definition of $C_T$, this proves that $S=0$, a contradiction.
\end{proof}

\end{document}